\documentclass[reqno]{amsart}
\usepackage[psamsfonts]{amssymb}
\usepackage{amsthm}
\usepackage{color}

\newtheorem{thm}{Theorem}

\newtheorem{cor}{Corollary}

\theoremstyle{remark}

\theoremstyle{definition}

 \theoremstyle{definition}
 
 \theoremstyle{remark}

 \numberwithin{equation}{section}

\def\bege{\begin{equation}} \def\ende{\end{equation}}

   \def\begr{\begin{eqnarray}}
\def\endr{\end{eqnarray}} 
 
\def\bege{\begin{equation}} \def\ende{\end{equation}}
\def\begr{\begin{eqnarray}} \def\endr{\end{eqnarray}}
\def\bnum{\begin{enumerate}} \def\enum{\end{enumerate}}

\allowdisplaybreaks[4]

\begin{document}

\title[Nuclear  Volterra composition operators]{Nuclear  Volterra composition operators between Bloch and weighted type spaces}
\thanks{The first author is thankful to  CSIR(India), Grant Number (File no. 09/1231(0001)/2019-EMR-I).\\
The second author is  thankful to  DST(SERB) for the project grant under MATRICS scheme. Grant No. MTR/2018/000479.}

\author[A. Sharma]{Aakriti
 Sharma}

\address{
Department of   Mathematics,
Central  University of Jammu,
Central University of Jammu, (Bagla) Rahya-Suchani, Samba-181143, J\&{K},
 India.}

\email{aakritishma321@gmail.com}

\author[A. K. Sharma]{Ajay
K. Sharma}

\address{%
Department of   Mathematics,
Central  University of Jammu,
Central University of Jammu, (Bagla) Raya-Suchani, Samba-181143, J\&{K},
 India.}

\email{aksju\_76@yahoo.com}

\maketitle

\begin{abstract}
In this paper, we  completely characterize nuclear Volterra composition operators  $T^\phi_g : \mathcal H^\infty_\nu \longrightarrow \mathcal H^\infty_\mu$ and
$S^\phi_g : \mathcal H^\infty_\nu \longrightarrow \mathcal H^\infty_\mu$
acting
 between weighted type spaces   in terms of the symbols $g$ and $\phi$  of $T^\phi_g$  and
$S^\phi_g$ and weights $\nu$ and $\mu$,  when the weights $\nu$ and $\mu$  are  normal weights in the sense of Shields and Williams. Moreover, nuclear Volterra composition operators
acting  between little weighted type spaces and Bloch spaces of order $\beta$ are also characterized.
\end{abstract}

\section{Introduction}
\label{intro} Denote  by ${H}(\mathbb{D})$, the class of all analytic functions on $\mathbb{D}$, where $\mathbb{D}$ is the open unit disk in the complex plane $\mathbb{C}$. For $g \in H(\mathbb D),$ Ch. Pommerenke  \cite{JHS} introduced  an integral operator $T_g : H(\mathbb D) \to H(\mathbb D)$ which maps every $f \in H(\mathbb D)$ to a function vanishing at $0,$ defined as
$$T_gf(z) = \int_{0}^{z}f(\zeta)dg(\zeta) = \int_{0}^{z}f(\zeta) g^{\prime}( z) d \zeta,\; \; \; z \in \mathbb D.$$  Another integral operator, generally known as  a companion operator of  $T_{g }$ was  defined by Yoneda   \cite{Yo} as $$ S_{g }f(z) =  \int_{0}^{z}f'(\zeta)  g(\zeta)d\zeta, \; \; \; z \in \mathbb D.$$ The operators $T_{g}$ and $S_{g}$ can be viewed as  cousins of Hankel and Toeplitz operators and they appeared in the literature under various names such as Volterra operators,  Cesaro operators, Riemann Stieltjes operators and  integration operators. 
For any $g \in {H}(\mathbb{D})$ and $\phi$ an analytic self-map of $\mathbb D$, the   Volterra composition operators
 $T^\phi_g$ and  $S^\phi_g$ are defined, respectively on ${H}(\mathbb{D})$ by
$$(T^\phi_g f)(z) =
\int_{0}^{z}
f(\phi(\zeta)) g'(\zeta) d\zeta, \; \; \; z \in \mathbb D $$ and $$(S^\phi_g f)(z) =
\int_{0}^{z}
f'(\phi(\zeta)) g(\zeta)d\zeta, \; \; \; z \in \mathbb D. $$ Recently, these type of operators are considered on spaces of analytic functions by several authors, see \cite{AlC01}-\cite{Ba14}, \cite{MDC},\cite{MDC1},\cite{ELP},\cite{LS}-\cite{QL1} and references therein. Motivated by results in \cite{JME} and \cite{FTL}, in this paper, we characterize nuclear  Volterra  composition  operators between Bloch and weighted type spaces of analytic functions,  when the weights  are  normal. 
 Recall that a radial weight $\nu$ is a non-increasing, non-negative continuous function defined on $\mathbb{D}$ such that $\nu(z)=\nu(|z|)$ for all $z\in \mathbb D.$   
 A radial weight $\nu$ is normal in the sense of Shields and Williams if it satisfies following two conditions:
\begin{enumerate}
\item [{($\mathbf{I}$).}] There exists  $\beta > 0$ such that the function $r\rightarrow\frac{\nu(r)}{(1-r^{2})^{\beta}}$ is almost increasing, or equivalently, $\inf_{n}\frac{\nu(1-2^{-(n+1)})}{\nu(1-2^{-n})}>0$.
\item [{($\mathbf{II}$).}] There exists $\gamma > 0$ such that  the function $r\rightarrow\frac{\nu(r)}{(1-r^{2})^{\gamma}}$ is almost decreasing, or equivalently, these is some $k \in \mathbb N$ such that $\limsup_{n}\frac{\nu(1-2^{-(n+k)})}{\nu(1-2^{-n})}<1$.
\end{enumerate}
It is clear that the standard weights $\nu_{\alpha}(z) = (1-|z|^{2})^\alpha, (\alpha > 0)$ are normal weights. Throughout this paper, we denote the  set of all normal weights on $\mathbb D$ by $N_W(\mathbb D).$  For more about normal weights, we refer to \cite{ASD},\cite{Ba14} and \cite{DJT1}.
The weighted Banach spaces $H^{\infty}_{\nu}$ and $H_{\nu}^{0}$ are defined as follows:
$$\mathcal H_{\nu}^{\infty} = \bigg\lbrace{f \in {H}(\mathbb{D}):\|f\|_{\mathcal H_{\nu}^{\infty}}=\sup_{z\in\mathbb{D}} \nu(z)|f(z)|  < \infty\bigg \rbrace}$$ and $$H_{\nu}^{0}=\bigg\lbrace{f \in  {H}(\mathbb{D}): \lim_{|z|\rightarrow 1^{-}}\nu(z)|f(z)|=0\bigg \rbrace}.$$
The weighted Bloch space $\mathcal B_\nu$ is the space of all functions $f$ in  ${H}(\mathbb{D})$ such that $ \|f\|_{\nu}=  \sup_{z\in \mathbb D} \nu(z) |f'(z)| < {\infty}.$ Note that $\|\cdot\|_{\nu}$ is a semi-norm on $\mathcal{B}_\nu$ and a norm can be defined on $\mathcal B_\nu$ as  \begin{equation*}\|f\|_{\mathcal B_\nu}= |f(0)| +\sup_{z\in \mathbb D} \nu(z) |f'(z)|.\end{equation*}
Endowed with the   norm $\|\cdot\|_{\mathcal B_\nu}$, $\mathcal B_\nu$ becomes a Banach space.   Similarly, the little weighted Bloch space $\mathcal B^0_{\nu }$ is  defined as  $\lim_{|z|\rightarrow 1^{-}} \nu(z) |f'(z)| =0.$
If $\lim_{|z|\rightarrow 1^{-}}\nu(z)=0$, then the weight $\nu$ is called typical and for typical weights $H_{\nu}^{\infty}$ is the bidual space of $H_{\nu}^{0}$. If the weight $\nu$ satisfies the condition ($\mathbf{I}$), then $\nu$ satisfies condition ($\mathbf{II}$) if and only if $H_{\nu}^{0}=\mathcal{B}^0_{\nu(z)(1-|z|^{2})}$, see \cite{WLY3}. In particular, if $\nu \in N_W(\mathbb D)$, then $H_{\nu}^{0}=\mathcal{B}_{\nu(z)(1-|z|^{2})}^{0}$ and by duality,
$H_{\nu}^{\infty} = \mathcal{B}_{\nu(z)(1-|z|^{2})}$, see  \cite{WLY}. Moreover, if   $\nu \in N_W(\mathbb D)$, then the topological isomorphisms $H_{\nu}^{0}\simeq c_{0}$ and $H_{\nu}^{\infty}\simeq l_{\infty}$ hold, see \cite{JME},\cite{ASD1}  and \cite{WLY}-\cite{WLY2}. Next, let us recall
a duality pairing  between $\mathcal H^\infty_\nu$ and a weighted  Bergman space, which was established by Shields and Williams in \cite{ASD}. Following \cite{ASD}, let $\nu$ be a typical weight   and $\omega$ a radial, positive and continuous function which satisfies $\int_{0}^{1}\omega(r)dr < \infty.$ Let
$A^1_{\omega}$ be the subspace of $H(\mathbb D)$ consists of functions $f$ such that $\|f\|_{A^1_{\omega}} = \int_{\mathbb D} |f(z)|\omega(z) dA(z) < \infty,$ where $dA$ is the normalized Lebesgue measure on $\mathbb D.$
For a normal weight $\nu$  and $\omega$ a function defined as above,  the  pair   $\{\nu, \omega\}$ is called a normal pair if \begin{equation}\label{1b1}\nu(r) \omega(r) = (1 - r^2)^\alpha, \; \; \; 0 \leq r < 1\end{equation} for some $\alpha > \beta -1$, where $\beta $ is a positive real number defined as in $(\mathbf{I}).$ For a  normal pair $\{\nu, \omega\}$, the following pairing between $\mathcal H^\infty_\nu$  and $A^1_{\omega}$ is defined in \cite{ASD}, see p. 292.
\begin{align} \label{1a1}
(f, g) & = \int_{\mathbb D} f(z) g(\bar{z})(1 - |z|^2)^\alpha dA(z)  = \int_{\mathbb D} f(\bar{z}) g({z})(1 - |z|^2)^\alpha dA(z).  \end{align}
 The pairing $(\cdot, \cdot)$ define a duality relation $(\mathcal H^0_\nu)' = A^1_{\omega},$ see \cite{ASD}, Theorem 2, p. 296. Thus we also have  $(\mathcal H^0_\nu)'' = \mathcal H^\infty_\nu.$ Moreover, if
 $K_\zeta(z) = (\alpha + 1)/(1 - {\zeta}z)^{\alpha + 2},$ $\alpha > -1$ is the kernel function and $\{\nu, \omega\}$ a normal pair, then $K_\zeta \in \mathcal H^0_\nu \cap A^1_{\omega},$ $g(\zeta) = (K_\zeta, g)$ for all $g \in A^1_{\omega}$ and $f(\zeta) = (f, K_\zeta)$ for all $f \in \mathcal H^\infty_\nu,$ see Lemma 10 in \cite{ASD}. \\

Recall that a linear operator $T$ between two Banach space $X$ and $Y$ is
\begin{enumerate}
\item[{(a)}] Nuclear if there is a sequence  $(x'_{k})\subset X'$ and  a sequence  $(y_{k})\subset Y$ such that $$\sum_{k}\|x'_{k}\|\|y_{k}\|<{\infty} \mbox{ and }  T = \sum_{k=1}^{\infty}x'_{k}\otimes y_{k},$$ where $x'_{k} \otimes y_{k}:X \rightarrow Y$ is  the mapping defined as $x\rightarrow x'_{k}(x)y_{k}$.
\item[{(b)}] Absolutely summing if there is $C>0$ such that for all finite sequences $(x_{i})_{i=1}^{n}\subset X$ we have $$\sum_{i=1}^{n}\|Tx_{i}\|\le C\sup_{\|x'\|\le1}\sum_{i=1}^{n}|x'(x_i)|= C \sup_{|\eta _{i}|=1}\|\sum_{i=1}^{n}\eta_{i}x_{i}\|.$$
\end{enumerate}
Moreover, it is well known that
any nuclear operator is compact and absolutely summing.
Also if $X= c_{0}$ or $l_{\infty}$, then a linear operator $T: X \rightarrow Y$ is absolutely summing if and only if it is nuclear, see \cite{FTL}.\\

Throughout this paper, $C$ denotes a positive constant, the exact value of which may  be not be same at each occurrence. The expression $A\lesssim B$ means that there is a  positive constant $C$ such that $A \le CB$ and the expression $A\asymp B$ means that there are  positive constants $C$ and $D$ such that $CB\le A\le DB$.
\section{Nuclear Volterra composition operators}\label{sec:prelimi}
     In this section, we characterize nuclear  Volterra composition operator between Bloch  and weighted  spaces of analytic functions. \\
     Next we present characterizations of bounded  Volterra composition operators between  weighted  spaces of analytic functions.  The arguments are standard and may possibly  appear in literature in a more general sense,  so we give an outline proof for completeness.
 \begin{thm} \label{th1} Let $\nu, \mu \in N_W(\mathbb D),$ $g \in H(\mathbb D)$ and $\phi$ a self-map of $\mathbb D$. Then   \begin{enumerate}
\item [{\rm($i$)}] The following are equivalent. \begin{enumerate}
\item[{\rm($a$)}]  $T_{g}^{\phi}: \mathcal H^{\infty}_{\nu} \longrightarrow  \mathcal H^{\infty}_{\mu}$ is bounded.
\item[{\rm($b$)}]  $T_{g}^{\phi}: \mathcal H^{0}_{\nu} \longrightarrow  \mathcal H^{\infty}_{\mu}$ is bounded.
\item[{\rm($c$)}]  $g$, $\nu,$ $ \mu$ and $\phi$ satisfy the following condition:  \begin{equation} \label{aak101}M_1 = \sup_{z\in \mathbb{D}}(1-|z|^{2}) \frac{\mu(z)}{\nu(\phi(z)) }|g'(z)| < \infty. \end{equation} \end{enumerate}
Moreover,   $T_{g}^{\phi}: \mathcal H^{0}_{\nu} \longrightarrow  \mathcal H^{0}_{\mu}$ is bounded if and only if
$ g \in \mathcal B^{0}_{(1-|z|^{2})\mu(z)} $ and {\rm (}\ref{aak101}{\rm)} holds.\item [{\rm($ii$)}] The following are equivalent.  \begin{enumerate}
\item[{\rm($d$)}]   $S_{g}^{\phi}: \mathcal H^{\infty}_{\nu} \longrightarrow  \mathcal H^{\infty}_{\mu}$ is bounded.\item[{\rm($e$)}]   $S_{g}^{\phi}: \mathcal H^{0}_{\nu} \longrightarrow  \mathcal H^{\infty}_{\mu}$ is bounded.
\item[{\rm($f$)}]  $g$, $\nu,$ $ \mu$ and $\phi$ satisfy the following condition:  \begin{equation} \label{aak102}\sup_{z\in \mathbb{D}}\frac{(1-|z|^{2})}{(1 - |\phi(z)|^2)}\frac{ \mu(z)}{ \nu(\phi(z)) }|g(z)| < \infty.\end{equation}  \end{enumerate} Moreover,   $S_{g}^{\phi}: \mathcal H^{0}_{\nu} \longrightarrow  \mathcal H^{0}_{\mu}$ is bounded if and only if
 $g \in \mathcal H^{0}_{(1-|z|^{2})\mu(z)} $  and {\rm (}\ref{aak102}{\rm)} holds.
\end{enumerate} \end{thm}
\begin{proof} $(i)$ Since $\mu$ is a normal weight, so $\mathcal H_{\mu(z)} = \mathcal{B}_{(1-|z|^{2})\mu(z)}.$ Thus boundedness of $T_{g}^{\phi}: \mathcal H^{\infty}_{\nu} \longrightarrow  \mathcal H^{\infty}_{\mu}$ is equivalent to the boundedness of $M_{g'}C_{\phi} : \mathcal H^{\infty}_{\nu} \longrightarrow  \mathcal H^{\infty}_{(1-|z|^{2})\mu(z)}.$ By Proposition 2 in \cite{DJT1}, every normal weight is essential, that is there is a constant $C > 0$ such that $\nu(z) \leq \tilde{\nu}(z) \leq C \nu(z)$ for each $z \in \mathbb D$.\\
$(a) \Leftrightarrow (c)$ 
 Then by Proposition 3.1 in  \cite{MDC2}, we have that $(a) \Leftrightarrow (c).$\\
 It is obvious that $(a) \Rightarrow (b).$ Thus to complete the proof, we need to show that  $(b) \Rightarrow (c).$\\
  $(b) \Rightarrow (c).$ Suppose that $M_{g'}C_{\phi} : \mathcal H^{0}_{\nu} \longrightarrow  \mathcal H^{\infty}_{(1-|z|^{2})\mu(z)}$ is bounded. For a fixed $\zeta \in \mathbb D,$ consider the function $ f_\zeta $ defined on $\mathbb D$ as
$$f_{\zeta}(z) = \frac{(1-|\phi(\zeta)|^2)^{\beta + 1}}{\nu(\phi(\zeta))(1- \overline{\phi(\zeta)}z)^{\beta+1}},$$ where $\beta > 0$ is as in condition
 $(\mathbf{I})$ of normal weight. Then it is easy to see that $ f_{\zeta} \in \mathcal H^{0}_{\nu} $ and $\sup_{\zeta \in \mathbb D}\| f_{\zeta}\|_{\mathcal H^{\infty}_{\nu}} \lesssim 1.$ Moreover, $f_{\zeta}(\phi(\zeta)) = 1/\nu(\phi(\zeta)).$
Thus by the boundedness of $M_{g'}C_{\phi} : \mathcal H^{\infty}_{\nu} \longrightarrow  \mathcal H^{\infty}_{(1-|z|^{2})\mu(z)}$, we have that  $$ (1-|\zeta|^{2})\mu(\zeta) |g'(\zeta)| |f_\zeta(\phi(\zeta))| \leq   \|M_{g'}C_{\phi}  f_\zeta \|_{\mathcal H^{\infty}_{(1-|z|^{2})\mu(z)}}  \lesssim \| f_\zeta \|_{\mathcal H^{\infty}_{\nu}} \lesssim 1. $$
 Taking supermum over $\zeta \in \mathbb D,$ we get (\ref{aak101}).\\
By Proposition 3.2 in  \cite{MDC2}, we have that  $T_{g}^{\phi}: \mathcal H^{0}_{\nu} \longrightarrow  \mathcal H^{0}_{\mu}$ is bounded if and only if
$ g \in \mathcal B^{0}_{(1-|z|^{2})\mu(z)} $ and {\rm (}\ref{aak101}{\rm)} holds.  This completes the proof of $(i)$.\\
$(ii)$  Thus boundedness of $S_{g}^{\phi}: \mathcal H^{\infty}_{\nu} \longrightarrow  \mathcal H^{\infty}_{\mu}$ is equivalent to the boundedness of $M_{g}C_{\phi}D : \mathcal H^{\infty}_{\nu} \longrightarrow  \mathcal H^{\infty}_{(1-|z|^{2})\mu(z)},$ where $D$ is the differentiation operator. We write an outline proof of  $(d) \Leftrightarrow (f)$ and $(e) \Leftrightarrow (f)$ The details of all other cases are omitted.\\
$(d) \Leftrightarrow (f)$ 
By Theorem 7 in  \cite{MZ1}, we have that $(d) \Leftrightarrow (f).$\\
 $(e) \Rightarrow (f).$ Suppose that $M_{g}C_{\phi}D : \mathcal H^{\infty}_{\nu} \longrightarrow  \mathcal H^{\infty}_{(1-|z|^{2})\mu(z)}$ is bounded.   For a fixed $\zeta \in \mathbb D,$ consider the function $ f_\zeta $ defined as in case $(i).$ Then proceeding as in case $(i)$, we have that   $f'_{\zeta}(\phi(\zeta)) = \overline{\phi(\zeta)}/(1-|\phi(\zeta)|^{2})\nu(\phi(\zeta)).$
Thus by the boundedness of $M_{g}C_{\phi}D : \mathcal H^{\infty}_{\nu} \longrightarrow  \mathcal H^{\infty}_{(1-|z|^{2})\mu(z)}$, we have that  \begin{equation*}  (1-|\zeta|^{2})\mu(\zeta) |g(\zeta)| |f'_\zeta(\phi(\zeta))| \leq   \|M_{g}C_{\phi}D  f_\zeta \|_{\mathcal H^{\infty}_{(1-|z|^{2})\mu(z)}}  \lesssim \| f_\zeta \|_{\mathcal H^{\infty}_{\nu}} \lesssim 1. \end{equation*}   Taking supermum over $\zeta \in \mathbb D,$ we get \begin{equation*}  \sup_{z\in \mathbb{D}}\frac{(1-|\zeta|^{2})\mu(\zeta)}{(1 - |\phi(\zeta)|^2)\nu(\phi(\zeta)) }|\phi(\zeta)| |g(\zeta)| < \infty.\end{equation*} Thus \begin{equation}\label{xxl}  \sup_{|\phi(\zeta)| > 1/2}\frac{(1-|\zeta|^{2})\mu(\zeta)}{(1 - |\phi(\zeta)|^2)\nu(\phi(\zeta)) }  |g(\zeta)| < \infty.\end{equation} The above condition easily yields (\ref{aak102}) by bifurcating the supermum over $z \in \mathbb D$ in (\ref{aak102}) as sum of supermum over $|\phi(\zeta)|\leq 1/2$ and  supermum over $|\phi(\zeta)| > 1/2$.  Thus the proof can be finished, using the fact that the first term in this sum is definitely finite and the second term in this sum is finite by (\ref{xxl}).\end{proof}

\begin{thm} \label{lm1}  Let $\nu, \mu \in N_W(\mathbb D).$ If $g \in H(\mathbb D)$ and $\phi$ a self-map of $\mathbb D$ such that $T_{g}^{\phi}: \mathcal H^{0}_{\nu} \longrightarrow  \mathcal H^{\infty}_{\mu}$ is bounded, then   \begin{enumerate}
\item [{\rm($i$)}] The following are equivalent. \begin{enumerate}
\item[{\rm($a$)}]  $T_{g}^{\phi}: \mathcal H^{0}_{\nu} \longrightarrow  \mathcal H^{\infty}_{\mu}$ is compact.
\item[{\rm($b$)}]  $T_{g}^{\phi}: \mathcal H^{0}_{\nu} \longrightarrow  \mathcal H^{\infty}_{\mu}$ is weakly compact.
\item[{\rm($c$)}]  The bi-transpose $(T_{g}^{\phi})''$ of $T_{g}^{\phi}: \mathcal H^{0}_{\nu} \longrightarrow  \mathcal H^{\infty}_{\mu}$  satisfies $(T^\phi_{g})'' = T^\phi_{g}$. \end{enumerate}
\item [{\rm($ii$)}] If $ g \in \mathcal B^{0}_{(1-|z|^{2})\mu(z)} $, then $T^\phi_{g}  : \mathcal H_{\nu}^{0} \rightarrow  \mathcal  H_{\mu}^{0}$  is bounded and the bi-transpose $(T^\phi_{g})''$ of $T^\phi_{g}$ satisfies $(T^\phi_{g})'' = T^\phi_{g}$. Moreover, $T^\phi_{g}  : \mathcal H_{\nu}^{0} \rightarrow  \mathcal  H_{\mu}^{0}$  is compact if and only if $T^\phi_{g}(H^{\infty}_{\nu}) \subset  \mathcal H^{0}_{\mu}.$
\end{enumerate} \end{thm}
\begin{proof} $(i)$ $(a) \Leftrightarrow (b)$ Since $ \mathcal H^{\infty}_{\mu} = \mathcal B_{(1 - |z|^2)\mu(z)}$, so $T_{g}^{\phi}f \in H^{\infty}_{\mu}$ if and only if $M_{g'}C_\phi f \in H^{\infty}_{(1 - |z|^2)\mu}.$ Moreover, $T_{g}^{\phi}f$ and $M_{g'}C_\phi f$ have comparable norm. Thus compactness (and/or weak compactness)  of $T_{g}^{\phi}: \mathcal H^{0}_{\nu} \longrightarrow  \mathcal H^{\infty}_{\mu}$   is equivalent to  compactness (and/or weak compactness) of $M_{g'}C_{\phi}: \mathcal H^{0}_{\nu} \longrightarrow  \mathcal H^{\infty}_{\mu}.$ By Theorem 1 in \cite{JME1}, $M_{g'}C_\phi f$ is compact if and only if it is weakly compact. Therefore, it follows that $(a) $ and $ (b)$ are equivalent.\\
$(a) \Rightarrow (c)$ Assume that $(a)$ holds. Then by Gantmacher-Nakamura's theorem, see Theorem 5.5 in \cite{Co90}, the bi-transpose  $(T_{g}^{\phi})''$ of $T_{g}^{\phi}: \mathcal H^{0}_{\nu} \longrightarrow  \mathcal H^{\infty}_{\mu}$ acts continuously from  $\mathcal H^{\infty}_{\nu}$ with $w^*$-topology to $\mathcal H^{\infty}_{\mu}$ with $w^*$-topology. Since pointwise topology $\tau_p$  is weaker than  $w^*$-topology on $\mathcal H^{\infty}_{\nu}$ and  $\mathcal H^{\infty}_{\mu}$, so $(T_{g}^{\phi})''$ and $T_{g}^{\phi}$ are both $w^*$ and $\tau_p$ continuous and they agree in   $\mathcal H^{0}_{\nu}$ which is $w^*$ dense in $\mathcal H^{\infty}_{\nu}.$ Therefore, $(c)$ holds.\\ $(c) \Rightarrow (b)$ Assume that the bi-transpose $(T_{g}^{\phi})''$ of $T_{g}^{\phi}: \mathcal H^{0}_{\nu} \longrightarrow  \mathcal H^{\infty}_{\mu}$  satisfies $(T^\phi_{g})'' = T^\phi_{g}$. Then it holds that $(T^\phi_{g})''(\mathcal H^{\infty}_{\nu}) \subset \mathcal H^{\infty}_{\mu}.$ Using the fact that a bounded linear operator $T$ acting from a Banach space $X$ to a Banach space $Y$ is weakly compact if and only if $T''(X'') \subset Y,$ see p. 482 in \cite{Du58}, we have that   $T_{g}^{\phi}: \mathcal H^{0}_{\nu} \longrightarrow  \mathcal H^{\infty}_{\mu}$ is weakly compact. \\
$(ii)$ If $ g \in \mathcal B^{0}_{(1-|z|^{2})\mu(z)} $, then by Theorem \ref{th1}, $T^\phi_{g}  : \mathcal H_{\nu}^{0} \rightarrow  \mathcal  H_{\mu}^{0}$  is bounded. Moreover, by  replacing $T_{g}$ in Lemma 1 in  \cite{Ba14} by $T^\phi_{g}$, we can easily prove that   $(T^\phi_{g})'' = T^\phi_{g}$. Finally, by  Gantmacher-Nakamura's theorem it holds that the compactness  of $T^\phi_{g}  : \mathcal H_{\nu}^{0} \rightarrow  \mathcal  H_{\mu}^{0}$  is equivalent to $T^\phi_{g}(H^{\infty}_{\nu}) \subset  \mathcal H^{0}_{\mu}.$
 This completes the proof.  \end{proof}
\begin{thm} \label{lm2}  Let $\nu, \mu \in N_W(\mathbb D).$ If $g \in H(\mathbb D)$ and $\phi$ a self-map of $\mathbb D$ such that $S_{g}^{\phi}: \mathcal H^{0}_{\nu} \longrightarrow  \mathcal H^{\infty}_{\mu}$ is bounded, then   \begin{enumerate}
\item [{\rm($i$)}] The following are equivalent. \begin{enumerate}
\item[{\rm($a$)}]  $S_{g}^{\phi}: \mathcal H^{0}_{\nu} \longrightarrow  \mathcal H^{\infty}_{\mu}$ is compact.
\item[{\rm($b$)}]  $S_{g}^{\phi}: \mathcal H^{0}_{\nu} \longrightarrow  \mathcal H^{\infty}_{\mu}$ is weakly compact.
\item[{\rm($c$)}]  The bi-transpose $(S_{g}^{\phi})''$ of $S_{g}^{\phi}: \mathcal H^{0}_{\nu} \longrightarrow  \mathcal H^{\infty}_{\mu}$  satisfies $(S^\phi_{g})'' = S^\phi_{g}$. \end{enumerate}
\item [{\rm($ii$)}] If $ g \in \mathcal H^{0}_{(1-|z|^{2})\mu(z)} $, then $S^\phi_{g}  : \mathcal H_{\nu}^{0} \rightarrow  \mathcal  H_{\mu}^{0}$  is bounded and the bi-transpose $(S^\phi_{g})''$ of $S^\phi_{g}$ satisfies $(S^\phi_{g})'' = S^\phi_{g}$. Moreover, $S^\phi_{g}  : \mathcal H_{\nu}^{0} \rightarrow  \mathcal  H_{\mu}^{0}$  is compact if and only if $S^\phi_{g}(H^{\infty}_{\nu}) \subset  \mathcal H^{0}_{\mu}.$
\end{enumerate} \end{thm} \begin{proof} (i) Proceeding as in Theorem \ref{lm1},   $S_{g}^{\phi}f \in H^{\infty}_{\mu}$ if and only if $M_{g'}C_\phi D f \in H^{\infty}_{(1 - |z|^2)\mu(z)}$ with comparable norm. Thus compactness (and/or weak compactness)  of $T_{g}^{\phi}: \mathcal H^{0}_{\nu} \longrightarrow  \mathcal H^{\infty}_{\mu}$   is equivalent to  compactness (and/or weak compactness) of $M_{g'}C_{\phi} D : \mathcal H^{0}_{\nu} \longrightarrow  \mathcal H^{\infty}_{\mu}$ which is further equivalent to is   compactness (and/or weak compactness) of $M_{g'}C_{\phi}  : \mathcal H^{0}_{(1 - |z|^2)\mu(z)} \longrightarrow  \mathcal H^{\infty}_{\mu}.$ Thus  the first part of theorem can now be settles as the proof of first part of  Theorem \ref{lm1}.\\
(ii) If $ g \in \mathcal H^{0}_{(1-|z|^{2})\mu(z)} $, then by Theorem \ref{th1}, $S^\phi_{g}  : \mathcal H_{\nu}^{0} \rightarrow  \mathcal  H_{\mu}^{0}$  is bounded. Moreover, by  replacing $S_{g}$ in Lemma 2.12 in  \cite{QL}, we can easily prove that   $(S^\phi_{g})'' = S^\phi_{g}$. The last part can be settle by Gantmacher-Nakamura's theorem as in the second part of  Theorem \ref{lm1}.\end{proof}
Finally, we recall  Piestsch's theorem which plays an important role in the proof of main results of this paper. For more about the Piestsch's theorem, we refer the readers to Theorem 2.12 of \cite{DJT}.
\begin{thm}\label{th2}(Pietsch's Theorem) Let $E$ and $F$ be Banach spaces and $1\leq p <+\infty$. A bounded linear operator $T: E \longrightarrow F$ is absolutely p-summing if and only if there is a constant $C$ and a regular Borel probability measure $\varrho$ on closed unit ball $B_{E^*}$ of $E^*$ with the weak$^*$-topology $\sigma (E^*, E)$ such that for every $x\in E$,$$\|T(x)\|\leq C\left(\int_{B_{E^{*}}}|\phi(x)|^{p}d\varrho(\phi)\right)^{1/p}.$$
\end{thm}
\begin{thm}\label{th3}  Let $\nu, \mu \in N_W(\mathbb D),$  $\alpha > -1$ and $\omega$ a weight function such that  $\{\nu, \omega\}$ is a normal pair and {\rm(}\ref{1b1}{\rm)} holds. If $g \in H(\mathbb D)$ and $\phi$ a self-map of $\mathbb D$ be such that  $T_{g}^{\phi}: \mathcal H^{\infty}_{\nu} \longrightarrow  \mathcal H^{\infty}_{\mu}$ is bounded, then the following statements are equivalent.
\begin{enumerate}
\item [{(a)}] $T_{g}^{\phi}: \mathcal H^{0}_{\nu} \longrightarrow  \mathcal H^{\infty}_{\mu}$ is nuclear.
\item [{(b)}] $T_{g}^{\phi}: \mathcal H^{\infty}_{\nu} \longrightarrow  \mathcal H^{\infty}_{\mu}$ is nuclear.
\item [{(c)}] $g$, $\nu,$ $\mu$ and $\phi $  satisfy the following condition:$$M_\alpha = \int_{\mathbb{D}}\bigg [\sup_{z\in \mathbb{D}}\frac{(1-|z|^{2})\mu(z)|g'(z)|}{|1-\bar{\zeta}\phi(z)|^{\alpha + 2}}\bigg ]\frac{(1-|\zeta|^{2})^\alpha}{\nu(\zeta)}dA(\zeta) < \infty.$$
\end{enumerate}
\end{thm}
\begin{proof}
$(c)\Rightarrow (a)$ Since $\mathcal H^{0}_{\nu} \simeq c_{0}$, so $T_{g}^{\phi}:  \mathcal H^{0}_{\nu} \longrightarrow \mathcal H^{\infty}_{\mu}$ is nuclear if and only if $T_{g}^{\phi}:  \mathcal H^{0}_{\nu} \longrightarrow \mathcal H^{\infty}_{\mu}$ is absolutely summing. Thus to complete the proof of the implication $(c)\Rightarrow (a)$, we need to  prove that $T_{g}^{\phi}:  \mathcal H^{0}_{\nu} \longrightarrow \mathcal H^{\infty}_{\mu}$ is absolutely summing. First note that if $\mu$ is normal,  $\mathcal H^\infty_{\mu(z)} = \mathcal{B}_{\mu(z)(1-|z|^{2})}$, $\|T_{g}^{\phi}(f)\|_{\mathcal H^{\infty}_{\mu}} \asymp \|T_{g}^{\phi}(f)\|_{\mathcal{B}_{\mu(z)(1-|z|^{2})}} $ for any $f \in H^\infty_{\nu}$ and $T_{g}^{\phi}(f)$ vanishes at origin for any $f \in H(\mathbb D)$.
Now polynomials are dense in $\mathcal H^{0}_{\mu}$, so it is sufficient to consider   polynomials $f_{1},f_{2},\cdots,f_{N}$ in the definition of absolutely summing.
Now for any $C > 1,$ we can select $z_{i}$, $i = 1, 2, \cdots N$ such that  \begin{align} \sum_{i=1}^{N}\|T_{g}^{\phi} f_i\|_{\mathcal H^\infty_{\mu}} &\asymp  \displaystyle\sum_{i=1
}^{N}\left(\big|(T_{g}^{\phi}f_i)(0)\big|+\sup_{z\in \mathbb{D}}(1-|z|^{2})\mu(z)\big|(T_{g}^{\phi} f_{i})'(z)\big|\right) \notag \\
&=  \sum_{i=1}^{N}\displaystyle\sup_{z\in \mathbb{D}}(1-|z|^{2})\mu(z) \big|f_{i}(\phi(z)) g'(z)\big| \notag \\
&\le  C\displaystyle\sum_{i=1}^{N}(1-|z_{i}|^{2})\mu(z_i) \big|f_{i}(\phi(z_{i})) g'(z_{i})\big| \label{nvcx1}\end{align}
Using the duality pairing in (\ref{1a1}), we have that
\begin{align*} \big|f_{i}(\phi(z_{i}))\big| = \big|(f_{i},  K_{\phi(z_{i})}) \big | = \bigg | (\alpha +1)\int_{\mathbb{D}} f_{i}(\zeta)\frac{(1-|\zeta|^{2})^\alpha  }{|1-\bar{\zeta}\phi(z_i)|^{\alpha + 2}} dA(\zeta) \bigg |. \end{align*}  Thus from (\ref{nvcx1}), we have that
\begin{align}
\sum_{i=1}^{N}\|T_{g}^{\phi} f_i\|_{_{\mathcal H^\infty_{\mu}}} & \leq   C \bigg (\sup_{w \in\mathbb{D}}\displaystyle\sum_{i=1}^{N}\big|f_{i}(w)\big| \nu(w) \bigg ) \notag \\ & \quad \times \int_{\mathbb{D}}\sup_{z\in\mathbb{D}}\frac{(1-|z|^{2})\mu(z)\big|g'(z)\big|}{\big|1-\bar{\zeta}\phi(z)\big|^{\alpha + 2}}\frac{(1-|\zeta|^{2})^\alpha}{\nu(\zeta)}dA(\zeta) \notag \\ & =  C \bigg (\sup_{w \in\mathbb{D}}\displaystyle\sum_{i=1}^{N}\big|f_{i}(w)\big| \nu(w) \bigg )M_\alpha\label{nvcw2}
 \end{align}
 Using the equality $$ \sup_{w \in\mathbb{D}} \sum_{i=1}^{N}  |f_{i}(w)|\nu(w)  = \sup_{|\eta_{i}| = 1}\sup_{w \in\mathbb{D}} \bigg|\displaystyle\sum_{i=1}^{N}\eta_{i}  f_{i}(w) \bigg | \nu(w) = \sup_{|\eta_{i}| = 1}\bigg \|\displaystyle\sum_{i=1}^{N}\eta_{i}  f_{i}\bigg \|_{\mathcal H^{\infty}_{\mu}}$$ in (\ref{nvcw2}), we have that
 \begin{align*} \sum_{i=1}^{N}\|T_{g}^{\phi} f_i\|_{_{\mathcal H^{\infty}_{\mu}}} & \leq   C \sup_{|\eta_{i}| = 1}\bigg \|\displaystyle\sum_{i=1}^{N}\eta_{i}  f_{i}\bigg \|_{\mathcal H^{\infty}_{\mu}}.
\end{align*}
Since $ C > 0$ is arbitrary, so $T_{g}^{\phi}:  \mathcal H^{0}_{\nu} \longrightarrow \mathcal H^{\infty}_{\mu}$ is absolutely summing.

$(a)\Rightarrow(c)$ Since every nuclear operator is absolutely summing, therefore, $T_{g}^{\phi}:  \mathcal H^{0}_{\nu} \longrightarrow \mathcal H^{\infty}_{\mu}$ is absolutely summing. By Theorem \ref{th2}, 
 there is a probability  Borel measure  $\varrho$ on $\sigma (A^{1}_\omega, \mathcal H^{0}_{\mu})$-compact unit ball $\mathbb B_1$ of $A^{1}_\omega$, where $\sigma (A^{1}_\omega, \mathcal H^{0}_{\mu})$ is the weak$^*$-topology on $(\mathcal H^{0}_{\mu})' = A^{1}_\omega,$ and some $\xi$ in $A^{1}_\omega$  such that
\begin{equation}\label{epu1}\|T_{g}^{\phi}f\|_{\mathcal H_{\mu}^{\infty}}\le C\int_{\mathbb B_1}|\xi(f)|d\varrho(\xi)\end{equation}
for every $f\in \mathcal H^{0}_{\mu}$ and some constant $C>0$ independent of $f$.
For every $\zeta \in \mathbb{D}$, we have that  $K_{{\zeta}} \in \mathcal H^{0}_{\mu}$. Therefore,
\begin{align}\label{epu2}
(\alpha + 1) \sup_{z\in\mathbb{D}}\frac{ (1-|z|^{2}) \mu(z)|g'(z)| }{|1- {\zeta}\phi(z)|^{\alpha + 2}}  \leq \|T_{g}^{\phi} K_{ {\zeta}}\|_{\mathcal B_{\mu(z)(1- |z|^2)}}\le C \|T_{g}^{\phi} K_{ {\zeta}}\|_{\mathcal H^{\infty}_{\mu(z) }}.\end{align}
Replacing $f$ in (\ref{epu1}) by $K_{{\zeta}}$, and then using  (\ref{epu2}), we have that
\begin{align*}
 \sup_{z\in\mathbb{D}}\frac{ (1-|z|^{2}) \mu(z)|g'(z)| }{|1- {\zeta}\phi(z)|^{\alpha + 2}}  \leq \frac{C}{\alpha + 1}  \int_{\mathbb B_1}|\xi(K_w)|d\varrho(\xi).\end{align*}
 Integrating  over $\mathbb{D}$ with respect to $\omega dA$, using (\ref{1b1}) and then applying Fubini's theorem, we get
\begin{align*}  \int_{\mathbb{D}}  \sup_{z\in\mathbb{D}}& \frac{ (1-|z|^{2}) \mu(z)|g'(z)| }{|1-\bar{\zeta}\phi(z)|^{\alpha + 2}} \frac{(1-|\zeta|^{2})^\alpha}{\nu(\zeta)} dA(\zeta) \notag \\
& \quad \quad \le \frac{C}{\alpha + 1} \int_{\mathbb B_1} \int_{\mathbb{D}}  |\xi(K_\zeta)| \omega(z) dA(\zeta) d\varrho(\xi). \end{align*} Now $K_\zeta \in \mathcal H^{0}_{\mu}$ and $\xi$ in $A^{1}_\omega$, so   $\xi(K_\zeta) = \langle K_\zeta, \xi \rangle =  \xi(w).$ Thus using the fact that $\varrho$ is a probability  Borel measure, we get \begin{align*}   \int_{\mathbb{D}}  \sup_{z\in\mathbb{D}}& \frac{ (1-|z|^{2}) \mu(z)|g'(z)| }{|1-\bar{\zeta}\phi(z)|^{\alpha + 2}} \frac{(1-|\zeta|^{2})^\alpha}{\nu(\zeta)} dA(\zeta) \notag \\
& \quad \quad \le \frac{C}{\alpha + 1} \sup_{\xi \in \mathbb B_1} \int_{\mathbb{D}}  |\xi(\zeta)| \omega(\zeta) dA(\zeta) \notag\\
& \quad \quad = \frac{C}{\alpha + 1} \sup_{\xi \in \mathbb B_1} \|\xi\|_{A^{1}_\omega} = \frac{C}{\alpha + 1}.
\end{align*}
This completes the proof of the implication $(a) \Rightarrow (c)$.\\
$(a)\Rightarrow(b)$ Since  transpose of a nuclear operator is nuclear, see  3.1.8 in [15], so by an application of Theorem \ref{lm1}(1), it follows that $(a)\Rightarrow(b)$.\\
$(b)\Rightarrow(a)$ Recall that  class  of nuclear operators  forms an operator ideal, that is, nuclear operators are a class of operators which do not fixes a copy of $l^\infty,$ (see \cite{JME1} and \cite{JME}),
and $\mathcal H^{0}_{\nu}$ is a closed subspace of $\mathcal H^{\infty}_{\mu}$, so it is obvious that $(b)\Rightarrow(a)$. This completes the proof. \end{proof}
\begin{thm} \label{1corA}   Let $\nu, \mu \in N_W(\mathbb D),$  $\alpha > -1$ and $\omega$ a weight function such that  $\{\nu, \omega\}$ is a normal pair and {\rm(}\ref{1b1}{\rm)} holds. If  $ g \in \mathcal B^{0}_{(1-|z|^{2})\mu(z)}$ and $\phi$ a self-map of $\mathbb D$   such that  $T_{g}^{\phi}: \mathcal H^{\infty}_{\nu} \longrightarrow  \mathcal H^{\infty}_{\mu}$ is bounded, then the following statements are equivalent.
\begin{enumerate}
\item [{(a)}] $T_{g}^{\phi}: \mathcal H^{0}_{\nu} \longrightarrow  \mathcal H^{\infty}_{\mu}$ is nuclear.
\item [{(b)}] $T_{g}^{\phi}: \mathcal H^{\infty}_{\nu} \longrightarrow  \mathcal H^{\infty}_{\mu}$ is nuclear.
\item [{(c)}] $T_{g}^{\phi}: \mathcal H^{0}_{\nu} \longrightarrow  \mathcal H^{0}_{\mu}$ is nuclear.
\item [{(d)}] $T_{g}^{\phi}: \mathcal H^{\infty}_{\nu} \longrightarrow  \mathcal H^{0}_{\mu}$ is nuclear.
\item [{(e)}] For  $g$, $\phi $ and $\nu$ satisfy the following condition:$$M_\alpha = \int_{\mathbb{D}}\bigg [\sup_{z\in \mathbb{D}}\frac{(1-|z|^{2})\mu(z)|g'(z)|}{|1-\bar{\zeta}\phi(z)|^{\alpha + 2}}\bigg ]\frac{(1-|\zeta|^{2})^\alpha}{\nu(\zeta)}dA(\zeta) < \infty.$$
\end{enumerate}
\end{thm}
\begin{proof} By Theorem \ref{th2}, we have that $(a) \Leftrightarrow (b) \Leftrightarrow (e).$ To complete the proof it is sufficient to prove that $(a) \Rightarrow (c) \Rightarrow (d)\Rightarrow (a).$\\
$(a) \Rightarrow (c)$ Suppose that $(a)$ holds. Then by Theorem \ref{lm1}(2)
$T_{g}^{\phi}: \mathcal H^{0}_{\nu} \longrightarrow  \mathcal H^{0}_{\mu}$ is bounded. Moreover, $T_{g}^{\phi}: \mathcal H^{0}_{\nu} \longrightarrow  \mathcal H^{\infty}_{\mu}$ is absolutely summing if and only if $T_{g}^{\phi}: \mathcal H^{0}_{\nu} \longrightarrow  \mathcal H^{0}_{\mu}.$ Thus the implication $(a) \Rightarrow (d)$ follows.\\
$(c) \Rightarrow (d)$  Using Theorem \ref{lm1}(ii) and the fact that transpose of a nuclear operator is nuclear the implication  follows.\\
$(d) \Rightarrow (a).$ The proof follows as the proof of  implication $(b) \Rightarrow (a)$ in Theorem \ref{th2}. This completes the proof.\end{proof}
\begin{cor}\label{2cor}  Let $\nu, \mu \in N_W(\mathbb D),$  $\alpha > -1$ and $\omega$ a weight function such that  $\{\nu, \omega\}$ is a normal pair and {\rm(}\ref{1b1}{\rm)} holds. If  $ g \in H(\mathbb D)$ and $\phi$ a self-map of $\mathbb D$   such that   $T_{g}^{\phi}: \mathcal H^{\infty}_{\nu} \longrightarrow  B_{\mu}$ is bounded, then the following statements are equivalent.
\begin{enumerate}
\item [{(a)}] $T_{g}^{\phi}: \mathcal H^{0}_{\nu} \longrightarrow  \mathcal B_{\mu}$ is nuclear.
\item [{(b)}] $T_{g}^{\phi}: \mathcal H^{\infty}_{\nu} \longrightarrow  \mathcal B_{\mu}$ is nuclear.
\item [{(c)}] For  $g$, $\phi $ and $\nu$ satisfy the following condition:$$ \int_{\mathbb{D}}\bigg [\sup_{z\in \mathbb{D}}\frac{ \mu(z)|g'(z)|}{|1-\bar{\zeta}\phi(z)|^{\alpha + 2}}\bigg ]\frac{(1-|\zeta|^{2})^\alpha}{\nu(\zeta)}dA(\zeta) < \infty.$$
\end{enumerate}
\end{cor}
\begin{cor}\label{3cor}  Let $\nu, \mu \in N_W(\mathbb D),$  $\alpha > -1$ and $\omega$ a weight function such that  $\{\nu, \omega\}$ is a normal pair and {\rm(}\ref{1b1}{\rm)} holds. If  $ g \in \mathcal B^{0}_{\mu}$ and $\phi$ a self-map of $\mathbb D$   such that  $T_{g}^{\phi}: \mathcal H^{\infty}_{\nu} \longrightarrow  \mathcal B_{\mu}$ is bounded, then the following statements are equivalent.
\begin{enumerate}
\item [{(a)}] $T_{g}^{\phi}: \mathcal H^{0}_{\nu} \longrightarrow  \mathcal B_{\mu}$ is nuclear.
\item [{(b)}] $T_{g}^{\phi}: \mathcal H^{\infty}_{\nu} \longrightarrow  \mathcal B_{\mu}$ is nuclear.
\item [{(c)}] $T_{g}^{\phi}: \mathcal H^{0}_{\nu} \longrightarrow  \mathcal B^{0}_{\mu}$ is nuclear.
\item [{(d)}] $T_{g}^{\phi}: \mathcal H^{\infty}_{\nu} \longrightarrow  \mathcal B^{0}_{\mu}$ is nuclear.
\item [{(e)}] For  $g$, $\phi $ and $\nu$ satisfy the following condition:$$ \int_{\mathbb{D}}\bigg [\sup_{z\in \mathbb{D}}\frac{\mu(z)|g'(z)|}{|1-\bar{\zeta}\phi(z)|^{\alpha + 2}}\bigg ]\frac{(1-|\zeta|^{2})^\alpha}{\nu(\zeta)}dA(\zeta) < \infty.$$
\end{enumerate}
\end{cor}
We can easily obtain the following corollaries.
\begin{cor}\label{4cor}  Let $\nu, \mu \in N_W(\mathbb D),$  $\alpha > -1$ and $\omega$ a weight function such that  $\{\nu, \omega\}$ is a normal pair and {\rm(}\ref{1b1}{\rm)} holds. If $\phi$ is a self-map of $\mathbb D$ such that  $C_{\phi}: \mathcal H^{\infty}_{\nu} \longrightarrow  \mathcal B_{\mu}$ is bounded, then the following statements are equivalent.
\begin{enumerate}
\item [{(a)}] $C_{\phi} : \mathcal H^{0}_{\nu} \longrightarrow  \mathcal B_{\mu}$ is nuclear.
\item [{(b)}] $C_{\phi} : \mathcal H^{\infty}_{\nu} \longrightarrow  \mathcal B_{\mu}$ is nuclear.
\item [{(c)}]   $\phi $ and $\nu$ satisfy the following condition:$$ \int_{\mathbb{D}}\bigg [\sup_{z\in \mathbb{D}}\frac{ \mu(z)|\phi'(z)|}{|1-\bar{\zeta}\phi(z)|^{\alpha + 2}}\bigg ]\frac{(1-|\zeta|^{2})^\alpha}{\nu(\zeta)}dA(\zeta) < \infty.$$
\end{enumerate}
\end{cor}
\begin{cor}\label{5cor}  Let $\nu, \mu \in N_W(\mathbb D),$  $\alpha > -1$ and $\omega$ a weight function such that  $\{\nu, \omega\}$ is a normal pair and {\rm(}\ref{1b1}{\rm)} holds. If   $ \phi \in \mathcal B^{0}_{\mu(z)} $ and  $C_{\phi}: \mathcal H^{\infty}_{\nu} \longrightarrow  \mathcal B_{\mu}$ is bounded, then the following statements are equivalent.
\begin{enumerate}
\item [{(a)}] $C_{\phi}: \mathcal H^{0}_{\nu} \longrightarrow  \mathcal B_{\mu}$ is nuclear.
\item [{(b)}] $C_{\phi}: \mathcal H^{\infty}_{\nu} \longrightarrow  \mathcal B_{\mu}$ is nuclear.
\item [{(c)}] $C_{\phi}: \mathcal H^{0}_{\nu} \longrightarrow  \mathcal B^{0}_{\mu}$ is nuclear.
\item [{(d)}] $C_{\phi}: \mathcal H^{\infty}_{\nu} \longrightarrow  \mathcal B^{0}_{\mu}$ is nuclear.
\item [{(e)}]   $\phi $ and $\nu$ satisfy the following condition:$$ \int_{\mathbb{D}}\bigg [\sup_{z\in \mathbb{D}}\frac{\mu(z)|\phi'(z)|}{|1-\bar{\zeta}\phi(z)|^{\alpha + 2}}\bigg ]\frac{(1-|\zeta|^{2})^\alpha}{\nu(\zeta)}dA(\zeta) < \infty.$$
\end{enumerate}
\end{cor}
\begin{cor}\label{2cor}   Let $\nu, \mu \in N_W(\mathbb D),$  $\alpha > -1$ and $\omega$ a weight function such that  $\{\nu, \omega\}$ is a normal pair and {\rm(}\ref{1b1}{\rm)} holds. If $g\in {H}(\mathbb D)$ be  such that  $T_{g} : \mathcal H^{\infty}_{\nu} \longrightarrow  B_{\mu}$ is bounded, then the following statements are equivalent.
\begin{enumerate}
\item [{(a)}] $T_{g} : \mathcal H^{0}_{\nu} \longrightarrow  \mathcal B_{\mu}$ is nuclear.
\item [{(b)}] $T_{g} : \mathcal H^{\infty}_{\nu} \longrightarrow  \mathcal B_{\mu}$ is nuclear.
\item [{(c)}] For  $g$  and $\nu$ satisfy the following condition:$$ \int_{\mathbb{D}}\bigg [\sup_{z\in \mathbb{D}}\frac{ \mu(z)|g'(z)|}{|1-\bar{\zeta}z|^{\alpha + 2}}\bigg ]\frac{(1-|\zeta|^{2})^\alpha}{\nu(\zeta)}dA(\zeta) < \infty.$$
\end{enumerate}
\end{cor}
\begin{cor}\label{3cor}   Let $\nu, \mu \in N_W(\mathbb D),$  $\alpha > -1$ and $\omega$ a weight function such that  $\{\nu, \omega\}$ is a normal pair and {\rm(}\ref{1b1}{\rm)} holds. If    $ g \in \mathcal B^{0}_{\mu(z)} $ be such that  $T_{g} : \mathcal H^{\infty}_{\nu} \longrightarrow  \mathcal B_{\mu}$ is bounded, then the following statements are equivalent.
\begin{enumerate}
\item [{(a)}] $T_{g} : \mathcal H^{0}_{\nu} \longrightarrow  \mathcal B_{\mu}$ is nuclear.
\item [{(b)}] $T_{g} : \mathcal H^{\infty}_{\nu} \longrightarrow  \mathcal B_{\mu}$ is nuclear.
\item [{(c)}] $T_{g} : \mathcal H^{0}_{\nu} \longrightarrow  \mathcal B^{0}_{\mu}$ is nuclear.
\item [{(d)}] $T_{g} : \mathcal H^{\infty}_{\nu} \longrightarrow  \mathcal B^{0}_{\mu}$ is nuclear.
\item [{(e)}] For  $g$  and $\nu$ satisfy the following condition:$$ \int_{\mathbb{D}}\bigg [\sup_{z\in \mathbb{D}}\frac{\mu(z)|g'(z)|}{|1-\bar{\zeta}z|^{\alpha + 2}}\bigg ]\frac{(1-|\zeta|^{2})^\alpha}{\nu(\zeta)}dA(\zeta) < \infty.$$
\end{enumerate}
\end{cor}
\begin{thm}\label{th4}  Let $\nu, \mu \in N_W(\mathbb D),$  $\alpha > -1$ and $\omega$ a weight function such that  $\{\nu, \omega\}$ is a normal pair and {\rm(}\ref{1b1}{\rm)} holds. If $ g \in H(\mathbb D)$ and $\phi$ a self-map of $\mathbb D$ such that $S_{g}^{\phi}: \mathcal H^{\infty}_{\nu} \longrightarrow  \mathcal H^{\infty}_{\mu}$ is bounded, then the following statements are equivalent.
\begin{enumerate}
\item [{(a)}] $S_{g}^{\phi}: \mathcal H^{0}_{\nu} \longrightarrow  \mathcal H^{0}_{\mu}$ is nuclear.
\item [{(b)}] $S_{g}^{\phi}: \mathcal H_{\nu} \longrightarrow  \mathcal H_{\mu}$ is nuclear.
\item [{(c)}] $g$, $\nu,$ $\mu$ and $\phi $ satisfy the following condition:$$N_\alpha = \int_{\mathbb{D}}\bigg [\sup_{z\in \mathbb{D}}\frac{(1-|z|^{2})\mu(z)|g(z)|}{|1-\bar{\zeta}\phi(z)|^{\alpha + 3}}\bigg ]\frac{(1-|\zeta|^{2})^\alpha}{\nu(\zeta)}dA(\zeta) < \infty.$$
\end{enumerate}
\end{thm}
\begin{proof} That $(a)\Leftrightarrow (b)$ can be proved, proceeding as in the proof of Theorem  \ref{th3}. Thus to complete the proof, we need to prove that
$(a)\Leftrightarrow(c)$.
$(c)\Rightarrow (a)$ Proceeding as in the proof  of $(c)\Rightarrow (a)$ of Theorem \ref{th3}, for any $C > 1,$ we can select $z_{i}$, $i = 1, 2, \cdots N$ such that   \begin{align} \sum_{i=1}^{N}\|T_{g}^{\phi} f_i\|_{\mathcal H_{\mu}^{\infty}}  \le  C\displaystyle\sum_{i=1}^{N}(1-|z_{i}|^{2})\mu(z_i) \big|f_{i}(\phi(z_{i})) g(z_{i})\big| \label{nvc1}\end{align}
and
\begin{align*}  f_{i}(w)  =  \langle f_{i},  K_{w}\rangle   =   (\alpha +1)\int_{\mathbb{D}} f_{i}(\zeta)\frac{(1-|\zeta|^{2})^\alpha  }{(1-\bar{\zeta}w)^{\alpha + 2}} dA(\zeta). \end{align*} Differentiating with respect to $w,$  we have that
\begin{align*}  f'_{i}(w)  =     (\alpha +1)(\alpha +2)\int_{\mathbb{D}} f_{i}(\zeta)\frac{\bar{\zeta}(1-|\zeta|^{2})^\alpha  }{(1-\bar{\zeta}w)^{\alpha + 3}} dA(\zeta). \end{align*}
Thus \begin{align} \label{mcv1} \big|f'_{i}(\phi(z_i))\big|   \leq  (\alpha +1)(\alpha +2)\int_{\mathbb{D}} |f_{i}(\zeta)|\frac{\bar{\zeta} (1-|\zeta|^{2})^\alpha  }{|1-\bar{\zeta}\phi(z_i)|^{\alpha + 3}} dA(\zeta). \end{align}  Using (\ref{nvc1}) and (\ref{mcv1} ), we have that
\begin{align}
\sum_{i=1}^{N}\|T_{g}^{\phi} f_i\|_{_{\mathcal H_{\mu}^{\infty}}} & \leq   C (\alpha +1)(\alpha +2)\bigg (\sup_{w \in\mathbb{D}}\displaystyle\sum_{i=1}^{N}\big|f_{i}(w)\big| \nu(w) \bigg ) \notag \\ & \quad \times \int_{\mathbb{D}}\sup_{z\in\mathbb{D}}\frac{(1-|z|^{2})\mu(z)\big|g'(z)\big|}{\big|1-\bar{\zeta}\phi(z)\big|^{\alpha + 3}}\frac{(1-|\zeta|^{2})^\alpha}{\nu(\zeta)}dA(\zeta) \notag \\ & =  C (\alpha +1)(\alpha +2)\bigg (\sup_{w \in\mathbb{D}}\displaystyle\sum_{i=1}^{N}\big|f_{i}(w)\big| \nu(w) \bigg )N_\alpha\label{nvc2}
 \end{align}  The proof can now be completed proceeding as in the Theorem \ref{th3}. We omit the details.\\
$(a)\Rightarrow(c)$ Once again, proceeding as in the proof of Theorem  \ref{th3},
 there exists a probability  Borel measure  $\varrho$ on $\sigma (A^{1}_\omega, \mathcal H^{0}_{\mu})$-compact unit ball $\mathbb B_1$ of $A^{1}_\omega$ and $\xi$ in $A^{1}_\omega$  such that
\begin{equation}\label{e1l}\|S_{g}^{\phi}f\|_{\mathcal H_{\mu}^{\infty}}\le C\int_{\mathbb B_1}|\xi(f)|d\varrho(\xi)\end{equation}
for every $f\in \mathcal H^{0}_{\mu}$ and some constant $C>0$ independent of $f$.
For every $\zeta \in \mathbb{D}$, we have that  $K_\zeta  \in \mathcal H^{0}_{\mu}$.
\begin{align}\label{e2l}
(\alpha + 1)(\alpha + 2) \sup_{z\in\mathbb{D}}\frac{ |\zeta| (1-|z|^{2}) \mu(z)|g(z)| }{|1-\bar{\zeta}\phi(z)|^{\alpha + 3}}  \leq \|S_{g}^{\phi} K_\zeta \|_{\mathcal B_{\mu(z)(1- |z|^2)}}\le C \|S_{g}^{\phi} K_\zeta \|_{\mathcal H_{\mu(z) }^{\infty}}.\end{align}
Replacing $f$ in (\ref{e1l}) by $K_\zeta $, and then using  (\ref{e2l}), we have that
\begin{align}\label{e3l}
 \sup_{z\in\mathbb{D}}\frac{ |\zeta| (1-|z|^{2}) \mu(z)|g(z)| }{|1-\bar{\zeta}\phi(z)|^{\alpha + 3}}   \leq \frac{C}{(\alpha + 1)(\alpha + 2)}  \int_{\mathbb B_1}|\xi(K_\zeta)|d\varrho(\xi).\end{align}
 Integrating (\ref{e3l}) over $\mathbb{D}$ with respect to $\omega dA$ and proceeding as in the Theorem  \ref{th3}, we get
\begin{align}\label{e33}  \int_{\mathbb{D}}  \sup_{z\in\mathbb{D}}\frac{ |\zeta| (1-|z|^{2}) \mu(z)|g(z)| }{|1-\bar{\zeta}\phi(z)|^{\alpha + 3}}  \frac{(1-|\zeta|^{2})^\alpha}{\nu(\zeta)} dA(\zeta)  \le \frac{C}{(\alpha + 1)(\alpha + 2)}.
\end{align}
Now if $ |\zeta| > 1/2,$ then by (\ref{e33}), we have that
\begin{align}\label{e3}  \int_{|\zeta| > 1/2}  \sup_{z\in\mathbb{D}}\frac{  (1-|z|^{2}) \mu(z)|g(z)| }{|1-\bar{\zeta}\phi(z)|^{\alpha + 3}}  \frac{(1-|\zeta|^{2})^\alpha}{\nu(\zeta)} dA(\zeta)  \le \frac{2C}{(\alpha + 1)(\alpha + 2)}.
\end{align}
Further, if $ |\zeta| \leq 1/2,$ then  we have that
\begin{align}\label{e3}  \int_{|\zeta| \leq 1/2}  \sup_{z\in\mathbb{D}}\frac{  (1-|z|^{2}) \mu(z)|g(z)| }{|1-\bar{\zeta}\phi(z)|^{\alpha + 3}}  \frac{(1-|\zeta|^{2})^\alpha}{\nu(\zeta)} dA(\zeta)  \le \frac{2^{\alpha+1}\|g\|_{\mathcal H_{\mu}^{\infty}}}{\nu(1/2)}.
\end{align} Thus
\begin{align*}\label{e4} N_\alpha &\leq \int_{|\zeta|\leq 1/2}\sup_{z\in\mathbb{D}}\frac{  (1-|z|^{2}) \mu(z)|g(z)| }{|1-\bar{\zeta}\phi(z)|^{\alpha + 3}}  \frac{(1-|\zeta|^{2})^\alpha}{\nu(\zeta)} dA(\zeta)\notag\\
& \quad + \int_{|\zeta|> 1/2} \sup_{z\in\mathbb{D}}\frac{  (1-|z|^{2}) \mu(z)|g(z)| }{|1-\bar{\zeta}\phi(z)|^{\alpha + 3}}  \frac{(1-|\zeta|^{2})^\alpha}{\nu(\zeta)} dA(\zeta)\notag\\
 &\leq \frac{2^{\alpha+1}\|g\|_{\mathcal H_{\mu}^{\infty}}}{\nu(1/2)} +\frac{2C}{(\alpha + 1)(\alpha + 2)}.
\end{align*}
This completes the proof. \end{proof}
\begin{cor}\label{1cor}  Let $\nu, \mu \in N_W(\mathbb D),$  $\alpha > -1$ and $\omega$ a weight function such that  $\{\nu, \omega\}$ is a normal pair and {\rm(}\ref{1b1}{\rm)} holds. If  $ g \in \mathcal H^{0}_{(1-|z|^{2})\mu(z)} $ and  $\phi$ an analytic self-map of $\mathbb{D}$ such that  $S_{g}^{\phi}: \mathcal H^{\infty}_{\nu} \longrightarrow  \mathcal H^{\infty}_{\mu}$ is bounded, then the following statements are equivalent.
\begin{enumerate}
\item [{(a)}] $S_{g}^{\phi}: \mathcal H^{0}_{\nu} \longrightarrow  \mathcal H^{\infty}_{\mu}$ is nuclear.
\item [{(b)}] $S_{g}^{\phi}: \mathcal H^{\infty}_{\nu} \longrightarrow  \mathcal H^{\infty}_{\mu}$ is nuclear.
\item [{(c)}] $S_{g}^{\phi}: \mathcal H^{0}_{\nu} \longrightarrow  \mathcal H^{0}_{\mu}$ is nuclear.
\item [{(d)}] $S_{g}^{\phi}: \mathcal H^{\infty}_{\nu} \longrightarrow  \mathcal H^{0}_{\mu}$ is nuclear.
\item [{(e)}] For  $g$, $\phi $ and $\nu$ satisfy the following condition:$$N_\alpha = \int_{\mathbb{D}}\bigg [\sup_{z\in \mathbb{D}}\frac{(1-|z|^{2})\mu(z)|g(z)|}{|1-\bar{\zeta}\phi(z)|^{\alpha + 3}}\bigg ]\frac{(1-|\zeta|^{2})^\alpha}{\nu(\zeta)}dA(\zeta) < \infty.$$
\end{enumerate}
\end{cor}
\begin{cor}\label{2cor}  Let $\nu, \mu \in N_W(\mathbb D),$  $\alpha > -1$ and $\omega$ a weight function such that  $\{\nu, \omega\}$ is a normal pair and {\rm(}\ref{1b1}{\rm)} holds. If   $g\in {H}(\mathbb D) and $  $\phi$ an analytic self-map of $\mathbb{D}$ such that  $S_{g}^{\phi}: \mathcal H^{\infty}_{\nu} \longrightarrow  B_{\mu}$ is bounded, then the following statements are equivalent.
\begin{enumerate}
\item [{(a)}] $S_{g}^{\phi}: \mathcal H^{0}_{\nu} \longrightarrow  \mathcal B_{\mu}$ is nuclear.
\item [{(b)}] $S_{g}^{\phi}: \mathcal H^{\infty}_{\nu} \longrightarrow  \mathcal B_{\mu}$ is nuclear.
\item [{(c)}] For  $g$, $\phi $ and $\nu$ satisfy the following condition:$$  \int_{\mathbb{D}}\bigg [\sup_{z\in \mathbb{D}}\frac{ \mu(z)|g(z)|}{|1-\bar{\zeta}\phi(z)|^{\alpha + 3}}\bigg ]\frac{(1-|\zeta|^{2})^\alpha}{\nu(\zeta)}dA(\zeta) < \infty.$$
\end{enumerate}
\end{cor}
\begin{cor}\label{3cor}   Let $\nu, \mu \in N_W(\mathbb D),$  $\alpha > -1$ and $\omega$ a weight function such that  $\{\nu, \omega\}$ is a normal pair and {\rm(}\ref{1b1}{\rm)} holds. If    $ g \in \mathcal B^{0}_{\mu(z)} $ and  $\phi$ an analytic self-map of $\mathbb{D}$ such that  $S_{g}^{\phi}: \mathcal H^{\infty}_{\nu} \longrightarrow  \mathcal B_{\mu}$ is bounded, then the following statements are equivalent.
\begin{enumerate}
\item [{(a)}] $S_{g}^{\phi}: \mathcal H^{0}_{\nu} \longrightarrow  \mathcal B_{\mu}$ is nuclear.
\item [{(b)}] $S_{g}^{\phi}: \mathcal H^{\infty}_{\nu} \longrightarrow  \mathcal B_{\mu}$ is nuclear.
\item [{(c)}] $S_{g}^{\phi}: \mathcal H^{0}_{\nu} \longrightarrow  \mathcal B^{0}_{\mu}$ is nuclear.
\item [{(d)}] $S_{g}^{\phi}: \mathcal H^{\infty}_{\nu} \longrightarrow  \mathcal B^{0}_{\mu}$ is nuclear.
\item [{(e)}] For  $g$, $\phi $ and $\nu$ satisfy the following condition:$$ \int_{\mathbb{D}}\bigg [\sup_{z\in \mathbb{D}}\frac{\mu(z)|g(z)|}{|1-\bar{\zeta}\phi(z)|^{\alpha + 3}}\bigg ]\frac{(1-|\zeta|^{2})^\alpha}{\nu(\zeta)}dA(\zeta) < \infty.$$
\end{enumerate}
\end{cor}
We can also obtain the following corollaries.
\begin{cor}\label{4cor}   Let $\nu, \mu \in N_W(\mathbb D),$  $\alpha > -1$ and $\omega$ a weight function such that  $\{\nu, \omega\}$ is a normal pair and {\rm(}\ref{1b1}{\rm)} holds. If    $\phi$ is an analytic self-map of $\mathbb{D}$ such that  $C_{\phi}D : \mathcal H^{\infty}_{\nu} \longrightarrow  B_{\mu}$ is bounded, then the following statements are equivalent.
\begin{enumerate}
\item [{(a)}] $C_{\phi}D : \mathcal H^{0}_{\nu} \longrightarrow  \mathcal B_{\mu}$ is nuclear.
\item [{(b)}] $C_{\phi}D : \mathcal H^{\infty}_{\nu} \longrightarrow  \mathcal B_{\mu}$ is nuclear.
\item [{(c)}]   $\phi $ and $\nu$ satisfy the following condition:$$  \int_{\mathbb{D}}\bigg [\sup_{z\in \mathbb{D}}\frac{ \mu(z)|\phi'(z)|}{|1-\bar{\zeta}\phi(z)|^{\alpha + 3}}\bigg ]\frac{(1-|\zeta|^{2})^\alpha}{\nu(\zeta)}dA(\zeta) < \infty.$$
\end{enumerate}
\end{cor}
\begin{cor}\label{5cor}   Let $\nu, \mu \in N_W(\mathbb D),$  $\alpha > -1$ and $\omega$ a weight function such that  $\{\nu, \omega\}$ is a normal pair and {\rm(}\ref{1b1}{\rm)} holds. If     $\phi$ an analytic self-map of $\mathbb{D}$ such that $ \phi \in \mathcal H^{0}_{\mu(z)} $ and  $C_{\phi} D : \mathcal H^{\infty}_{\nu} \longrightarrow  \mathcal B_{\mu}$ is bounded, then the following statements are equivalent.
\begin{enumerate}
\item [{(a)}] $C_{\phi}D : \mathcal H^{0}_{\nu} \longrightarrow  \mathcal B_{\mu}$ is nuclear.
\item [{(b)}] $C_{\phi}D : \mathcal H^{\infty}_{\nu} \longrightarrow  \mathcal B_{\mu}$ is nuclear.
\item [{(c)}] $C_{\phi}D : \mathcal H^{0}_{\nu} \longrightarrow  \mathcal B^{0}_{\mu}$ is nuclear.
\item [{(d)}] $C_{\phi}D : \mathcal H^{\infty}_{\nu} \longrightarrow  \mathcal B^{0}_{\mu}$ is nuclear.
\item [{(e)}]   $\phi $ and $\nu$ satisfy the following condition:$$ \int_{\mathbb{D}}\bigg [\sup_{z\in \mathbb{D}}\frac{\mu(z)|\phi'(z)|}{|1-\bar{\zeta}\phi(z)|^{\alpha + 3}}\bigg ]\frac{(1-|\zeta|^{2})^\alpha}{\nu(\zeta)}dA(\zeta) < \infty.$$
\end{enumerate}
\end{cor}
\begin{cor}\label{2cor}  Let $\nu, \mu \in N_W(\mathbb D),$  $\alpha > -1$ and $\omega$ a weight function such that  $\{\nu, \omega\}$ is a normal pair and {\rm(}\ref{1b1}{\rm)} holds. If   Let $g\in {H}(\mathbb D)$ be  such that  $S_{g} : \mathcal H^{\infty}_{\nu} \longrightarrow  B_{\mu}$ is bounded, then the following statements are equivalent.
\begin{enumerate}
\item [{(a)}] $S_{g} : \mathcal H^{0}_{\nu} \longrightarrow  \mathcal B_{\mu}$ is nuclear.
\item [{(b)}] $S_{g} : \mathcal H^{\infty}_{\nu} \longrightarrow  \mathcal B_{\mu}$ is nuclear.
\item [{(c)}] For  $g$  and $\nu$ satisfy the following condition:$$ \int_{\mathbb{D}}\bigg [\sup_{z\in \mathbb{D}}\frac{ \mu(z)|g'(z)|}{|1-\bar{\zeta}z|^{\alpha + 3}}\bigg ]\frac{(1-|\zeta|^{2})^\alpha}{\nu(\zeta)}dA(\zeta) < \infty.$$
\end{enumerate}
\end{cor}
\begin{cor}\label{3cor}   Let $\nu, \mu \in N_W(\mathbb D),$  $\alpha > -1$ and $\omega$ a weight function such that  $\{\nu, \omega\}$ is a normal pair and {\rm(}\ref{1b1}{\rm)} holds. If    $ g \in \mathcal B^{0}_{\mu(z)} $ be such that  $S_{g} : \mathcal H^{\infty}_{\nu} \longrightarrow  \mathcal B_{\mu}$ is bounded, then the following statements are equivalent.
\begin{enumerate}
\item [{(a)}] $S_{g} : \mathcal H^{0}_{\nu} \longrightarrow  \mathcal B_{\mu}$ is nuclear.
\item [{(b)}] $S_{g} : \mathcal H^{\infty}_{\nu} \longrightarrow  \mathcal B_{\mu}$ is nuclear.
\item [{(c)}] $S_{g} : \mathcal H^{0}_{\nu} \longrightarrow  \mathcal B^{0}_{\mu}$ is nuclear.
\item [{(d)}] $S_{g} : \mathcal H^{\infty}_{\nu} \longrightarrow  \mathcal B^{0}_{\mu}$ is nuclear.
\item [{(e)}]   $g$  and $\nu$ satisfy the following condition:$$ \int_{\mathbb{D}}\bigg [\sup_{z\in \mathbb{D}}\frac{\mu(z)|g'(z)|}{|1-\bar{\zeta}z|^{\alpha + 3}}\bigg ]\frac{(1-|\zeta|^{2})^\alpha}{\nu(\zeta)}dA(\zeta) < \infty.$$
\end{enumerate}
\end{cor}
\section{Nuclear $T_{g}^{\phi}$ and $S_{g}^{\phi}$  between   Bloch spaces of order $\beta$}
Finally, we apply our results to $T_{g}^{\phi}$ and $S_{g}^{\phi}$ acting between  weighted Banach spaces and Bloch spaces of order $\beta,$ where $\beta > 0.$ Recall that if $\nu_\beta$ is a classical weight, that is $\nu_\beta (z) = (1 - |z|^2)^\beta$, then the spaces $\mathcal H^{\infty}_{\nu}$, $\mathcal H^{0}_{\nu}$, $\mathcal B_{\nu}$  and $\mathcal B^{0}_{\mu}$ reduces respectively to weighted and Bloch spaces of order $\beta$, denoted by $\mathcal H^{\infty}_{\beta}$, $\mathcal H^{0}_{\beta}$, $\mathcal B_{\beta}$  and $\mathcal B^{0}_{\beta}$. Moreover, for $\beta  > 1$, we have that $\mathcal H^{0}_{\beta - 1} = \mathcal B^{0}_{\beta}$ and $\mathcal H^{\infty}_{\beta - 1} = \mathcal B^{\infty}_{\beta}$ with equivalent norm. Thus we have the following corollary. \begin{cor}\label{1thm}  Let $\beta > 1$, $\gamma  > 0$, $g\in {H}(\mathbb D)$ and $\phi$ an analytic self-map of $\mathbb{D}$ such that  $T_{g}^{\phi}: \mathcal B_{\beta} \longrightarrow  \mathcal B_{\gamma}$ is bounded. Then the following statements are equivalent.
\begin{enumerate}
\item [{(a)}] $T_{g}^{\phi}: \mathcal B^0_{\beta} \longrightarrow  \mathcal B_{\gamma}$ is nuclear.
\item [{(b)}] $T_{g}^{\phi}: \mathcal B_{\beta} \longrightarrow  \mathcal B_{\gamma}$ is nuclear.
\item [{(c)}] $g$, $\beta,$  $\gamma $ and $\phi$ satisfy the following condition: $$P_\alpha = \int_{\mathbb{D}}\bigg [\sup_{z\in \mathbb{D}}\frac{(1-|z|^{2})^{\gamma}|g'(z)|}{|1-\bar{\zeta}\phi(z)|^{\alpha + 2}}\bigg ] {(1-|\zeta|^{2})^{\alpha - \beta + 1}}dA(\zeta) < \infty.$$
\end{enumerate}
\end{cor} The case $\beta  = 1$ is settled in the next theorem.
 \begin{thm}\label{1t}  Let $g\in {H}(\mathbb D)$ and   $\phi$ an analytic self-map of $\mathbb{D}$ such that  $T_{g}^{\phi}: \mathcal B  \longrightarrow  \mathcal B $ is bounded.   Then the following statements are equivalent.
\begin{enumerate}
\item [{(a)}] $T_{g}^{\phi}: \mathcal B_{0}\longrightarrow  \mathcal B$ is nuclear.
\item [{(b)}] $g$ and $\phi $ satisfy the following condition:$$M = \int_{\mathbb{D}}\sup_{z\in \mathbb{D}}\frac{(1-|z|^{2})|\phi'(z)|}{|1-\bar{w}\phi(z)|^{3}}|g(z)|dA(w) < \infty.$$
\end{enumerate}
\end{thm}
\begin{proof}
$(b)\Rightarrow (a)$ Since $c_{0}\simeq \mathcal B_{0}$ and polynomials are dense in $ \mathcal B_{0}$. Moreover, by  Corollary $1.5$ in \cite{HKZ}, we have that $f  \in \mathcal  B,$   then we have that
\begin{equation}\label{1ee}
  f'(z)=2\int_{\mathbb{D}}\frac{1-|w|^2}{(1-z\bar{w})^3}f'(w)dA(w).
\end{equation}
Consider  polynomials $p_{1},p_{2},\cdots,p_{N}$ in the definition of absolutely summing and then using (\ref{1ee}), we have
$$ \sum_{i=1}^{N}\|T_{g}^{\phi}p_i\|_{ \mathcal B}  \le  C \bigg(\sup_{w \in\mathbb{D}}\displaystyle\sum_{i=1}^{N}\big|p'_{i}(w)\big|(1-|w|^{2}) \bigg )\int_{\mathbb{D}}\sup_{z\in\mathbb{D}}\frac{(1-|z|^{2})\big|\phi'(z)\big|}{\big|1-\bar{w}\phi(z)\big|^{3}}\big|g(z)\big|dA(w).
$$ The proof can be settled  proceeding as in the proof of Theorem \ref{th3}.\\
$(a)\Rightarrow(b)$  By Theorem 1, 
 there exists a probability  Borel measure  $\varrho$ on $\sigma ((\mathcal B_{0})^{*}, \mathcal B)$-compact unit ball $B'$ of $(\mathcal B_{0})^{*}=A^{1}$ such that
\begin{equation*} \|T_{g}^{\phi}f\|_{\mathcal B}\le C\int_{B^{'}}|\xi(f)|d\varrho(\xi)\end{equation*}
for every $f\in \mathcal B_{0}$ and some constant $C>0$ independent of $f$.
Next for every $w\in \mathbb{D}$, we consider $f_{w}(z)= {1}/{(1-\bar{w}z)^{2}}$ which lies in $ \mathcal B_{0}\cap H^{\infty}$. By duality and reproducing property of $f_{w}$ in the Bergman space $( \mathcal B_{0})^{*}=A^{1}$, we have
\begin{equation*} \xi(f_{w})=\langle f_{w},h\rangle=\int_{\mathbb{D}}h(z)\overline{f_{w}(z)}dA(z)=\int_{\mathbb{D}}\dfrac{h(z)}{(1-\bar{z}w)^{2}}dA(z)=h(w).\end{equation*}
Also,\begin{equation*}       \sup_{z\in\mathbb{D}}\dfrac{2|w|(1-|z|^{2})|\phi'(z)||g(z)|}{|1-\bar{w}\phi(z)|^{3}}=\|T_{g}^{\phi}f_w\|_{1}\leq\|T_{g}^{\phi}f_w\|_{\mathcal{B}}
     \end{equation*} Now proceeding as in the proof of Theorem \ref{th4},
we can easily see that $(b)$ holds. This completes the proof.\end{proof}
Similarly, we have the following corollary.
\begin{cor}\label{1thm}    Let $\beta > 1$, $\gamma  > 0$, $g\in {H}(\mathbb D)$ and   $\phi$ an analytic self-map of $\mathbb{D}$ such that    $S_{g}^{\phi}: \mathcal B_{\beta} \longrightarrow  \mathcal B_{\gamma}$ is bounded. Then the following statements are equivalent.
\begin{enumerate}
\item [{(a)}] $S_{g}^{\phi}: \mathcal B^0_{\beta} \longrightarrow  \mathcal B_{\gamma}$ is nuclear.
\item [{(b)}] $S_{g}^{\phi}: \mathcal B_{\beta} \longrightarrow  \mathcal B_{\gamma}$ is nuclear.
\item [{(c)}] For  $g$, $\phi $ and $\nu$ satisfy the following condition:$$Q_\alpha = \int_{\mathbb{D}}\bigg [\sup_{z\in \mathbb{D}}\frac{(1-|z|^{2})^{\gamma}|g(z)|}{|1-\bar{\zeta}\phi(z)|^{\alpha + 3}}\bigg ] {(1-|\zeta|^{2})^{\alpha - \beta + 1}}dA(\zeta) < \infty.$$
\end{enumerate}
\end{cor}

\end{document}